\documentclass[12pt,reqno,twoside]{article}
\usepackage{amsfonts,amssymb,amsmath,amsthm}
\usepackage[top=20mm,bottom=20mm,left=20mm,right=20mm]{geometry}
\usepackage[T1]{fontenc}
\usepackage{times}
\usepackage{multicol}
\usepackage{amsfonts}
\usepackage{graphicx}
\usepackage{graphics}
\usepackage{enumitem}
\usepackage{cite}
\allowdisplaybreaks

\newtheorem{theorem}{Theorem}
\newtheorem{definition}[theorem]{Definition}

\theoremstyle{definition}
\newtheorem{example}[theorem]{Example}
\begin{document}
\title{Tauberian theorems for weighted means of double sequences in intuitionistic fuzzy normed spaces}
\author{Lakshmi Narayan Mishra$^1$\footnote{Corresponding author. E-mail: lakshminarayanmishra04@gmail.com}\, Mohd. Raiz$^2$, Vishnu Narayan Mishra$^2$}

\date{{\small
		$^1$Department of Mathematics, School of Advanced Sciences,\\ Vellore Institute of Technology (VIT) University,\\ Vellore 632 014, Tamil Nadu, India\\ Email: lakshminarayanmishra04@gmail.com\\
		$^2$Department of Mathematics,\\ Indira Gandhi National Tribal University,\\Lalpur, Amarkantak, Anuppur, Madhya Pradesh 484 887, India\\ Email: mohdraizp@gmail.com, vnm@igntu.ac.in, vishnunarayanmishra@gmail.com\\}}
\maketitle

\begin{abstract}
We define weighted mean summability method of double sequences in intuitionistic fuzzy normed spaces($IFNS$), and obtain necessary and sufficient Tauberian conditions under which convergence of double sequences in $IFNS$ follows from their weighted mean summability. This study reveals also Tauberian results for some known summation methods in the special cases.
\end{abstract}

\section{Introduction}
Zadeh\cite{zadeh} introduced the concept of fuzzy sets by assigning membership grades to each element in the universe of discourse in order to handle noncategorical and unclassifiable data. Following its introduction, fuzzy sets were studied by many mathematicians and applied to many fields of science. Besides, Atanassov\cite{atanassov} introduced intuitionistic fuzzy sets(IFS) by also assigning nonmembership grades to the elements. IFS were also studied by many mathematicians. The concept of intuitionistic fuzzy norm was studied by Saadati and Park\cite{ifnorm}, and by Lael\cite{lael}. Convergence of sequences in $IFNS$ was studied and some convergence and summation methods were introduced to recover the convergence where ordinary convergence of sequences in $IFNS$ fails\cite{karakus,mursaleen,taloyavuz,yavuz}. In this study, we extend the results of \cite{chen} to intuitionistic fuzzy normed spaces. That is, we define weighted mean summability methods for double sequences in $IFNS$ and give some Tauberian conditions under which convergence of double sequences in $IFNS$ follows from weighted mean summability. This study also reveals Tauberian results for some known summation methods such as Ces\`{a}ro summability method $(C,1,1)$ and N\"{o}rlund summability method $(\bar{N},p)$ in the special cases. We now give some preliminaries for $IFNS$.
\begin{definition}\cite{lael}
The triplicate $(V,\mu,\nu)$ is said to be an $IF-$normed space if $V$ is a real vector space, and $\mu,\nu$ are $F-$sets on $V\times\mathbb{R}$ satisfying the following conditions for every $x,y\in V$ and $t,s\in\mathbb{R}$:
\begin{enumerate}[label=(\alph*)]
\item $\mu(x,t)=0$ for all non-positive real number t,
\item $\mu(x,t)=1$ for all $t\in\mathbb{R^+}$ if and only if $x=\theta$
\item $\mu(cx,t)=\mu\left(x,\frac{t}{|c|}\right)$ for all $t\in\mathbb{R^+}$ and $c\neq0$,
\item $\mu(x+y,t+s)\geq \min\{\mu(x,t),\mu(y,s)\}$,
\item $\lim_{t\to\infty}\mu(x,t)=1$ and $\lim_{t\to0}\mu(x,t)=0$,

\item $\nu(x,t)=1$ for all non-positive real number t,
\item $\nu(x,t)=0$ for all $t\in\mathbb{R^+}$ if and only if $x=\theta$
\item $\nu(cx,t)=\nu\left(x,\frac{t}{|c|}\right)$ for all $t\in\mathbb{R^+}$ and $c\neq0$,
\item $\max\{\nu(x,t),\nu(y,s)\}\geq\nu(x+y,t+s)$,
\item $\lim_{t\to\infty}\nu(x,t)=0$ and $\lim_{t\to0}\nu(x,t)=1$.
\end{enumerate}
In this case, we will call $(\mu,\nu)$ an $IF-$norm on $V$.
\end{definition}
Throughout the paper $(V,\mu,\nu)$ will denote an $IF-$normed space.
\begin{example}\label{standartnorm}
Let $(V, \Vert\cdot\Vert)$ be a normed space and $\mu_0$, $\nu_0$ be $F-$sets on $V\times\mathbb{R}$ defined by
\begin{eqnarray*}
\mu_0(x,t)=
\begin{cases}
0, \quad &t\leq0,\\
\frac{t}{t+\Vert x\Vert}, &t>0,
\end{cases}
\hspace{2cm}
\nu_0(x,t)=
\begin{cases}
1, \quad &t\leq0,\\
\frac{\Vert x\Vert}{t+\Vert x\Vert}, &t>0.
\end{cases}
\end{eqnarray*}
Then $(\mu_0, \nu_0)$ is $IF-$norm on $V$.
\end{example}
\begin{definition}\label{convergence}\cite{ifnorm}
A double sequence $(x_{mn})$ in $(V,\mu,\nu)$ is said to be convergent to $x\in V$ and denoted by $x_{mn}\to x$, if for each $t>0$ and each $\varepsilon \in(0,1)$ there exists $n_0\in \mathbb{N}$ such that
\begin{equation*}
    \mu(x_{mn}-x,t)>1-\varepsilon \quad \textrm{and} \quad \nu(x_{mn}-x,t)<\varepsilon
\end{equation*}
for all $m,n\geq n_0$.
\end{definition}
Here we note that convergence of double sequences in $IFNS$ is meant in the sense in Definition \ref{convergence} throughout the paper. Similarly, limit and convergence of double sequences of real numbers are meant in the Pringsheim's sense\cite{pringsheim}.
\begin{definition}\cite{ifnorm}
A double sequence $(x_{mn})$ in $(V,\mu,\nu)$ is said to be Cauchy if for each $t>0$ and each $\varepsilon \in(0,1)$ there exists $n_0\in \mathbb{N}$ such that
\begin{equation*}
    \mu(x_{jk}-x_{mn},t)>1-\varepsilon \quad \textrm{and} \quad \nu(x_{jk}-x_{mn},t)<\varepsilon
\end{equation*}
for all $i,j,m,n\geq n_0$.
\end{definition}
We note that if sequence $(x_{mn})$ converges to $x\in V$, then $(x_{mn})$ is Cauchy in view of the facts that
\begin{eqnarray*}
\mu(x_{jk}-x_{mn},t)\geq\min\{\mu(x_{jk}-x,t/2),\mu(x_{mn}-x,t/2)\}\\ \nu(x_{jk}-x_{mn},t)\leq\max\{\nu(x_{jk}-x,t/2),\nu(x_{mn}-x,t/2)\}.
\end{eqnarray*}
\begin{definition}\label{q-bounded}\cite{qbounded}
A double sequence $(x_{mn})$ in $(V,\mu,\nu)$ is called q-bounded if $\lim\limits_{t\rightarrow\infty}\inf\limits_{m,n}\mu(x_{mn},t)=1$ and $\lim\limits_{t\rightarrow\infty}\sup\limits_{m,n}\nu(x_{mn},t)=0$.
\end{definition}
\begin{theorem}\label{boundedtoslowly}\cite{taloyavuz}
Let $(x_n)$ be a sequence in $(V,\mu,\nu)$. If $\{n(x_{n}-x_{n-1})\}$ is q-bounded, then $(x_n)$ is slowly oscillating.
\end{theorem}
Let $p=(p_j)$ and $q=(q_k)$ be two sequences of nonnegative numbers($p_0,q_0>0$) with
\begin{eqnarray*}
P_m=\sum_{j=0}^{m}p_j\to\infty\  (m\to\infty),\qquad Q_n=\sum_{k=0}^{n}q_k\to\infty\  (n\to\infty).
\end{eqnarray*}
Let $(\alpha,\beta)\in\{(1,1),(1,0),(0,1)\}$. The weighted means $(t^{\alpha\beta}_{mn})$ of a double sequence $(x_{mn})$ in $IF-$normed space $(V,\mu,\nu)$ are defined by
\begin{eqnarray*}
t^{11}_{mn}=\frac{1}{P_mQ_n}\sum_{j=0}^{m}\sum_{k=0}^{n}p_jq_kx_{jk},\quad t^{10}_{mn}=\frac{1}{P_m}\sum_{j=0}^{m}p_jx_{jn},\quad t^{01}_{mn}=\frac{1}{Q_n}\sum_{k=0}^{n}q_kx_{mk}.
\end{eqnarray*}
The double sequence $(x_{mn})$ in $(V,\mu,\nu)$ is said to be $(\bar{N},p,q;\alpha,\beta)$ summable to $x\in V$ if $\lim_{m,n\to\infty}t^{\alpha\beta}_{mn}=x$, and denoted by $x_{mn}\to x\ (\bar{N},p,q;\alpha,\beta)$.

Since $t^{10}_{mn}$ is independent of $q$, we will write $(\bar{N},p,*;1,0)$ in place of $(\bar{N},p,q;1,0)$(see \cite{chen}). We note that if we take $p_j=1,q_k=1$ for all $j,k\in\mathbb{N}$ in $(\bar{N},p,q;1,1)$ summability, then we obtain $(C,1,1)$ summability. Also if we take $(x_{jk})=(x_{j})$ in $(\bar{N},p,*;1,0)$ summability, then we obtain $(\bar{N},p)$ summability of single sequences. Hence, this paper reveals also Tauberian results for $(C,1,1)$ and $(\bar{N},p)$ summability methods in special cases. Here we note that when the case is $(x_{jk})=(x_{j})$ in $(\bar{N},p,*;1,0)$ summability, the condition of q-boundedness in Theorem \ref{regular1} will be redundant.
\section{Results for $(\bar{N},p,q;1,1)$ summability in $IFNS$}
In this section we give some Tauberian conditions under which $(\bar{N},p,q;1,1)$ summability implies convergence in $IFNS$. Before to give our Tauberian results, we first show that convergence and q-boundedness imply $(\bar{N},p,q;1,1)$ summability for double sequences in $IFNS$.
\begin{theorem}\label{regular}
If double sequence $(x_{mn})$ in $(V,\mu,\nu)$ is q-bounded and convergent to $x\in V$, then $(x_{mn})$ is $(\bar{N},p,q;1,1)$ summable to $x$.
\end{theorem}
\begin{proof}
Let double sequence $(x_{mn})$ in $(V,\mu,\nu)$ be q-bounded and convergent to $x\in V$. Fix $t>0$. Let $n_0$ be from Definition \ref{convergence} of convergence with $t\leftrightarrow\frac{t}{3}$. Then, we have
{\small\begin{eqnarray*}
\mu\left(t^{11}_{mn}-x,t\right)&=&\mu\left(\frac{1}{P_mQ_n}\sum_{j=0}^{m}\sum_{k=0}^{n}p_jq_kx_{jk}-x,t\right)
=\mu\left(\sum_{j=0}^{m}\sum_{k=0}^{n}p_jq_k(x_{jk}-x),P_mQ_nt\right)
\\&\geq&
\min\left\{\mu\left(\sum_{j=0}^{m}\sum_{k=0}^{n_0}p_jq_k(x_{jk}-x),P_mQ_nt/3\right), \mu\left(\sum_{j=0}^{n_0}\sum_{k=n_0+1}^{n}p_jq_k(x_{jk}-x),P_mQ_nt/3\right),\right.
\\&&
\qquad\qquad\left.\mu\left(\sum_{j=n_0+1}^{m}\sum_{k=n_0+1}^{n}p_jq_k(x_{jk}-x),P_mQ_nt/3\right)\right\}
\\&\geq&
\min\left\{\min_{\substack{j\in\mathbb{N}\\0\leq k\leq n_0}}\mu\left(x_{jk}-x,\frac{Q_n}{n_0A}\frac{t}{3}\right), \min_{\substack{k\in\mathbb{N}\\0\leq j\leq n_0}}\mu\left(x_{jk}-x,\frac{P_m}{n_0B}\frac{t}{3}\right), \min_{n_0< k,j}\mu\left(x_{jk}-x,\frac{t}{3}\right)\right\}
\\&\geq&
\min\left\{\inf_{j,k\in\mathbb{N}}\mu\left(x_{jk}-x,\frac{Q_n}{n_0A}\frac{t}{3}\right), \inf_{j,k\in\mathbb{N}}\mu\left(x_{jk}-x,\frac{P_m}{n_0B}\frac{t}{3}\right), \min_{n_0<j,k}\mu\left(x_{jk}-x,\frac{t}{3}\right)\right\}
\end{eqnarray*}}
and
\begin{eqnarray*}
\nu\left(t^{11}_{mn}-x,t\right)&=&\nu\left(\sum_{j=0}^{m}\sum_{k=0}^{n}p_jq_k(x_{jk}-x),P_mQ_nt\right)
\\&\leq&
\max\left\{\sup_{j,k\in\mathbb{N}}\nu\left(x_{jk}-x,\frac{Q_n}{n_0A}\frac{t}{3}\right), \sup_{j,k\in\mathbb{N}}\nu\left(x_{jk}-x,\frac{P_m}{n_0B}\frac{t}{3}\right), \max_{n_0< k,j}\nu\left(x_{jk}-x,\frac{t}{3}\right)\right\}.
\end{eqnarray*}
where $A=\max_{0\leq k\leq n_0}q_k, B=\max_{0\leq j\leq n_0}p_j$. Also, by q-boundedness of $(x_{mn})$ there exists $n_1\in\mathbb{N}$ such that
\begin{eqnarray*}
\inf_{j,k\in\mathbb{N}}\mu\left(x_{jk}-x,\frac{Q_n}{n_0A}\frac{t}{3}\right)>1-\varepsilon,\quad\sup_{j,k\in\mathbb{N}}\nu\left(x_{jk}-x,\frac{Q_n}{n_0A}\frac{t}{3}\right)<\varepsilon \\
\inf_{j,k\in\mathbb{N}}\mu\left(x_{jk}-x,\frac{P_m}{n_0B}\frac{t}{3}\right)>1-\varepsilon,\quad \sup_{j,k\in\mathbb{N}}\nu\left(x_{jk}-x,\frac{P_m}{n_0B}\frac{t}{3}\right)<\varepsilon
\end{eqnarray*}
for $n,m>n_1$. Hence, combining all above we get
\begin{eqnarray*}
\mu\left(t^{11}_{mn}-x,t\right)>1-\varepsilon,\qquad \nu\left(t^{11}_{mn}-x,t\right)<\varepsilon
\end{eqnarray*}
for $m,n>\max\{n_0,n_1\}$, which implies $t^{11}_{mn}\to x$ by Definition \ref{convergence}.
\end{proof}
The converse statement of Theorem \ref{regular} is not valid which can be seen by the next example. That is, $(\bar{N},p,q;1,1)$ summability does not imply convergence in $IFNS$.
\begin{example}
Let us consider $IFNS$ $(\mathbb{R},\mu_0,\nu_0)$ where $\mu_0,\nu_0$ are from Example \ref{standartnorm}. Sequence $(x_{mn})$ defined by $x_{mn}=(-1)^{m+n}$ is in $(\mathbb{R},\mu_0,\nu_0)$. $(x_{mn})$ is $(\bar{N},p,q;1,1)$ summable to 0 with $p_j=1,q_k=1(j,k\in\mathbb{N})$, but it is not convergent in $(\mathbb{R},\mu_0,\nu_0)$. See also \cite[Example 3.3]{taloyavuz}.
\end{example}
Our aim is to give the conditions under which $(\bar{N},p,q;1,1)$ summability implies convergence in $IFNS$.

Let $SV\!A_+$ be the set of all nonnegative sequences $p=(p_j)$ with $p_0>0$ satisfying
\begin{eqnarray*}
\liminf_{m\to\infty}\left|\frac{P_{\lambda_m}}{P_m}-1\right|>0\quad for \ all\ \lambda>0 \ with\ \lambda\neq1
\end{eqnarray*}
where $\lambda_m=[\lambda m]$ and $[\cdot]$ denotes integral part\cite{chen}. In \cite[Lemma 2.2]{chen}, equivalent assertions for the set $SV\!A_+$ are obtained.
\begin{theorem}\label{theorem}
Let $p,q\in SV\!A_+$ and double sequence $(x_{mn})$ be in $(V,\mu,\nu)$. Assume that $x_{mn}\to x\ (\bar{N},p,q;1,1)$. Then, $x_{mn}\to x$ if and only if for all $t>0$
\begin{eqnarray}\label{tauber1}
\sup_{\lambda>1}\liminf_{m,n\rightarrow\infty}\mu\left(\frac{1}{(P_{\lambda_m}-P_m)(Q_{\lambda_n}-Q_n)}\sum_{j=m+1}^{\lambda_m}\sum_{k=n+1}^{\lambda_n}p_jq_k(x_{jk}-x_{mn}),t\right)=1
\end{eqnarray}
and
\begin{eqnarray}\label{tauber2}
\inf_{\lambda>1}\limsup_{m,n\rightarrow\infty}\nu\left(\frac{1}{(P_{\lambda_m}-P_m)(Q_{\lambda_n}-Q_n)}\sum_{j=m+1}^{\lambda_m}\sum_{k=n+1}^{\lambda_n}p_jq_k(x_{jk}-x_{mn}),t\right)=0.
\end{eqnarray}
\end{theorem}
\begin{proof}
{\it Necessity.} Let $x_{mn}\to x$.  Fix $t>0$. For any $\lambda>1$, and for sufficiently large $m,n$ such that $\lambda_m>m, \lambda_n>n$ we have(see\cite[Eqation (3.3)]{chen})
\begin{eqnarray}\label{basiceq}
&&\frac{1}{(P_{\lambda_m}-P_m)(Q_{\lambda_n}-Q_n)}\sum_{j=m+1}^{\lambda_m}\sum_{k=n+1}^{\lambda_n}p_jq_k(x_{jk}-x_{mn})=
\\&& t^{11}_{\lambda_m,\lambda_n}-x_{mn}+\frac{1}{\frac{P_{\lambda_m}}{P_m}-1}(t^{11}_{\lambda_m,\lambda_n}-t^{11}_{m,\lambda_n})+\frac{1}{\frac{Q_{\lambda_n}}{Q_n}-1}(t^{11}_{\lambda_m,\lambda_n}-t^{11}_{\lambda_m,n})\nonumber
\\&&
+\frac{1}{\left(\frac{P_{\lambda_m}}{P_m}-1\right)}\frac{1}{\left(\frac{Q_{\lambda_n}}{Q_n}-1\right)}(t^{11}_{\lambda_m,\lambda_n}-t^{11}_{m,\lambda_n}-t^{11}_{\lambda_m,n}+t^{11}_{mn})\nonumber
\end{eqnarray}
Since $p,q\in SV\!A_+$ and $t^{11}_{mn}\to x$ we have
\begin{eqnarray*}
\mu\left(t^{11}_{\lambda_m,\lambda_n}-x_{mn},t\right)\to1,\quad \nu\left(t^{11}_{\lambda_m,\lambda_n}-x_{mn},t\right)\to0\ as\qquad \min\{m,n\}\to\infty,
\end{eqnarray*}
{\small\begin{eqnarray}\label{needed1}
\begin{split}
\mu\left(\frac{1}{\frac{P_{\lambda_m}}{P_m}-1}(t^{11}_{\lambda_m,\lambda_n}-t^{11}_{m,\lambda_n}),t\right)\geq\mu\left(t^{11}_{\lambda_m,\lambda_n}-t^{11}_{m,\lambda_n},\left\{\liminf_{m\to\infty} \left|\frac{P_{\lambda_m}}{P_m}-1\right|\right\}t\right)\to1\ as\ \min\{m,n\}\to\infty
\\
\nu\left(\frac{1}{\frac{P_{\lambda_m}}{P_m}-1}(t^{11}_{\lambda_m,\lambda_n}-t^{11}_{m,\lambda_n}),t\right)\leq\nu\left(t^{11}_{\lambda_m,\lambda_n}-t^{11}_{m,\lambda_n},\left\{\liminf_{m\to\infty} \left|\frac{P_{\lambda_m}}{P_m}-1\right|\right\}t\right)\to0\ as\ \min\{m,n\}\to\infty
\end{split}
\end{eqnarray}}
and
{\small\begin{eqnarray}\label{needed2}
\begin{split}
\mu\left(\frac{1}{\frac{Q_{\lambda_m}}{Q_m}-1}(t^{11}_{\lambda_m,\lambda_n}-t^{11}_{\lambda_m,n}),t\right)\geq\mu\left(t^{11}_{\lambda_m,\lambda_n}-t^{11}_{\lambda_m,n},\left\{\liminf_{m\to\infty} \left|\frac{Q_{\lambda_m}}{Q_m}-1\right|\right\}t\right)\to1\ as\ \min\{m,n\}\to\infty
\\
\nu\left(\frac{1}{\frac{Q_{\lambda_m}}{Q_m}-1}(t^{11}_{\lambda_m,\lambda_n}-t^{11}_{\lambda_m,n}),t\right)\leq\nu\left(t^{11}_{\lambda_m,\lambda_n}-t^{11}_{\lambda_m,n},\left\{\liminf_{m\to\infty} \left|\frac{Q_{\lambda_m}}{Q_m}-1\right|\right\}t\right)\to0\ as\ \min\{m,n\}\to\infty
\end{split}
\end{eqnarray}}
we obtain
\begin{eqnarray*}
\lim_{m,n\to\infty}\mu\left(\frac{1}{(P_{\lambda_m}-P_m)(Q_{\lambda_n}-Q_n)}\sum_{j=m+1}^{\lambda_m}\sum_{k=n+1}^{\lambda_n}p_jq_k(x_{jk}-x_{mn}),t\right)=1
\\ \lim_{m,n\to\infty}\nu\left(\frac{1}{(P_{\lambda_m}-P_m)(Q_{\lambda_n}-Q_n)}\sum_{j=m+1}^{\lambda_m}\sum_{k=n+1}^{\lambda_n}p_jq_k(x_{jk}-x_{mn}),t\right)=0.
\end{eqnarray*}
which implies \eqref{tauber1} and \eqref{tauber2}.

{\it Sufficiency.} Let \eqref{tauber1} and \eqref{tauber2} be satisfied. Fix $t>0$.  Then for given $\varepsilon>0$ we have:
\begin{itemize}
\item There exist $\lambda>1$ and $n_0\in\mathbb{N}$ such that
\begin{eqnarray*}
\mu\left(\frac{1}{(P_{\lambda_m}-P_m)(Q_{\lambda_n}-Q_n)}\sum_{j=m+1}^{\lambda_m}\sum_{k=n+1}^{\lambda_n}p_jq_k(x_{jk}-x_{mn}),t/5\right)>1-\varepsilon
\\
\nu\left(\frac{1}{(P_{\lambda_m}-P_m)(Q_{\lambda_n}-Q_n)}\sum_{j=m+1}^{\lambda_m}\sum_{k=n+1}^{\lambda_n}p_jq_k(x_{jk}-x_{mn}),t/5\right)<\varepsilon
\end{eqnarray*}
for all $m,n>n_0$.
\item There exists $n_1\in\mathbb{N}$ such that $\mu\left(t^{11}_{\lambda_m,\lambda_n}-x,t/5\right)>1-\varepsilon$ and $\nu\left(t^{11}_{\lambda_m,\lambda_n}-x,t/5\right)<\varepsilon$ for $m,n>n_1$.
\item There exists $n_2\in\mathbb{N}$ such that
\begin{equation*}
    \mu\left(\frac{1}{\frac{P_{\lambda_m}}{P_m}-1}(t^{11}_{\lambda_m,\lambda_n}-t^{11}_{m,\lambda_n}),t/5\right)>1-\varepsilon
    \ \textrm{and}\
    \nu\left(\frac{1}{\frac{P_{\lambda_m}}{P_m}-1}(t^{11}_{\lambda_m,\lambda_n}-t^{11}_{m,\lambda_n}),t/5\right)<\varepsilon
\end{equation*}
for $m,n>n_2$ by \eqref{needed1}.
\item There exists $n_3\in\mathbb{N}$ such that
\begin{equation*}
\mu\left(\frac{1}{\frac{Q_{\lambda_m}}{Q_m}-1}(t^{11}_{\lambda_m,\lambda_n}-t^{11}_{\lambda_m,n}),t/5\right)>1-\varepsilon
   \ \textrm{and}\
\nu\left(\frac{1}{\frac{Q_{\lambda_m}}{Q_m}-1}(t^{11}_{\lambda_m,\lambda_n}-t^{11}_{\lambda_m,n}),t/5\right)<\varepsilon
\end{equation*}
for $m,n>n_3$ by \eqref{needed2}.
\item There exists $n_4\in\mathbb{N}$ such that
\begin{eqnarray*}
\mu\left(\frac{1}{\left(\frac{P_{\lambda_m}}{P_m}-1\right)}\frac{1}{\left(\frac{Q_{\lambda_n}}{Q_n}-1\right)}(t^{11}_{\lambda_m,\lambda_n}-t^{11}_{m,\lambda_n}-t^{11}_{\lambda_m,n}+t^{11}_{mn}),t/5\right)>1-\varepsilon
\\
\nu\left(\frac{1}{\left(\frac{P_{\lambda_m}}{P_m}-1\right)}\frac{1}{\left(\frac{Q_{\lambda_n}}{Q_n}-1\right)}(t^{11}_{\lambda_m,\lambda_n}-t^{11}_{m,\lambda_n}-t^{11}_{\lambda_m,n}+t^{11}_{mn}),t/5\right)<\varepsilon
\end{eqnarray*}
for $m,n>n_4$.
\end{itemize}
Then, by equation \eqref{basiceq} we get
\begin{eqnarray*}
&&\mu\left(x_{mn}-x,t\right)\\
&\geq&\min\left\{\mu\left(t^{11}_{\lambda_m,\lambda_n}-x,t/5\right),\mu\left(\frac{1}{(P_{\lambda_m}-P_m)(Q_{\lambda_n}-Q_n)}\sum_{j=m+1}^{\lambda_m}\sum_{k=n+1}^{\lambda_n}p_jq_k(x_{jk}-x_{mn}),t/5\right),\right.
\\&&\qquad\qquad
\mu\left(\frac{1}{\frac{P_{\lambda_m}}{P_m}-1}(t^{11}_{\lambda_m,\lambda_n}-t^{11}_{m,\lambda_n}),t/5\right),\mu\left(\frac{1}{\frac{Q_{\lambda_m}}{Q_m}-1}(t^{11}_{\lambda_m,\lambda_n}-t^{11}_{\lambda_m,n}),t/5\right),
\\&&\qquad\qquad
\left.\mu\left(\frac{1}{\left(\frac{P_{\lambda_m}}{P_m}-1\right)}\frac{1}{\left(\frac{Q_{\lambda_n}}{Q_n}-1\right)}(t^{11}_{\lambda_m,\lambda_n}-t^{11}_{m,\lambda_n}-t^{11}_{\lambda_m,n}+t^{11}_{mn}),t/5\right)\right\}
\\&>&1-\varepsilon
\end{eqnarray*}
and, similarly,
\begin{eqnarray*}
\nu\left(x_{mn}-x,t\right)<\varepsilon
\end{eqnarray*}
for $m,n>\max\{n_0,n_1,n_2,n_3,n_4\}$, and this implies $x_{mn}\to x$. The proof is completed.
\end{proof}
We can give following theorem similar to Theorem \ref{theorem}. The proof is done in a similar way by making the changes $\lambda_m\leftrightarrow m$ and $\lambda_n\leftrightarrow n$.
\begin{theorem}\label{theorem2}
Let $p,q\in SV\!A_+$ and double sequence $(x_{mn})$ be in $(V,\mu,\nu)$. Assume that $x_{mn}\to x\ (\bar{N},p,q;1,1)$. Then, $x_{mn}\to x$ if and only if for all $t>0$
\begin{eqnarray*}
\sup_{0<\lambda<1}\liminf_{m,n\rightarrow\infty}\mu\left(\frac{1}{(P_m-P_{\lambda_m})(Q_n-Q_{\lambda_n})}\sum_{j=\lambda_m+1}^{m}\sum_{k=\lambda_n+1}^{n}p_jq_k(x_{mn}-x_{jk}),t\right)=1
\end{eqnarray*}
and
\begin{eqnarray*}
\inf_{0<\lambda<1}\limsup_{m,n\rightarrow\infty}\nu\left(\frac{1}{(P_m-P_{\lambda_m})(Q_n-Q_{\lambda_n})}\sum_{j=\lambda_m+1}^{m}\sum_{k=\lambda_n+1}^{n}p_jq_k(x_{mn}-x_{jk}),t\right)=0.
\end{eqnarray*}
\end{theorem}
Now we define slow oscillation for double sequences in $IFNS$.
\begin{definition}
A double sequence $(x_{mn})$ in $(V,\mu,\nu)$ is said to be slowly oscillating in the sense (1,1) if
\begin{equation*}
\sup_{\lambda>1}\liminf_{m,n\rightarrow\infty}\min_{\substack{m<j\leq\lambda_m\\ n<k\leq\lambda_n}}\mu(x_{jk}-x_{mn},t)=1
\end{equation*}
and
\begin{equation*}
\inf_{\lambda>1}\limsup_{m,n\rightarrow\infty}\max_{\substack{m<j\leq\lambda_m\\ n<k\leq\lambda_n}}\nu(x_{jk}-x_{mn},t)=0 ,
\end{equation*}
for all $t>0$.
\end{definition}
A double sequence $(x_{mn})$ is slowly oscillating in the sense (1,1) if and only if for all $t>0$ and for all $\varepsilon\in (0,1)$ there exist
$\lambda>1$ and $n_0\in \mathbb{N}$, depending on $t$ and $\varepsilon$, such that
\begin{equation*}
    \mu(x_{jk}-x_{mn},t)>1-\varepsilon \quad \textrm{and} \quad \nu(x_{jk}-x_{mn},t)<\varepsilon
\end{equation*}
whenever $n_0\leq m< j\leq\lambda_m$ and $n_0\leq n< k\leq\lambda_n$.
\begin{theorem}\label{slowtauber}
Let $p,q\in SV\!A_+$ and $x\in V$. If double sequence $(x_{mn})$ in $(V,\mu,\nu)$ is slowly oscillating in the sense (1,1) and $x_{mn}\to x\ (\bar{N},p,q;1,1)$, then $x_{mn}\to x$.
\end{theorem}
\begin{proof}
Let $x_{mn}\to x\ (\bar{N},p,q;1,1)$ and $(x_{mn})$ be slowly oscillating in the sense (1,1). Fix $t>0$. Then for given $\varepsilon\in (0,1)$ there exist $\lambda>1$ and $n_0\in \mathbb{N}$ such that
\begin{equation*}
     \mu(x_{jk}-x_{mn},t)>1-\varepsilon \quad \textrm{and} \quad \nu(x_k-x_n,t)<\varepsilon
\end{equation*}
whenever $n_0\leq m< j\leq\lambda_m$ and $n_0\leq n< k\leq\lambda_n$ by slow oscillation. Hence, we have
\begin{eqnarray*}
\mu\left(\frac{1}{(P_{\lambda_m}-P_m)(Q_{\lambda_n}-Q_n)}\sum_{j=m+1}^{\lambda_m}\sum_{k=n+1}^{\lambda_n}p_jq_k(x_{jk}-x_{mn}),t\right)
\geq
\min_{\substack{m<j\leq\lambda_m\\ n<k\leq\lambda_n}}\mu(x_{jk}-x_{mn},t)>1-\varepsilon
\end{eqnarray*}
and
\begin{eqnarray*}
\nu\left(\frac{1}{(P_{\lambda_m}-P_m)(Q_{\lambda_n}-Q_n)}\sum_{j=m+1}^{\lambda_m}\sum_{k=n+1}^{\lambda_n}p_jq_k(x_{jk}-x_{mn}),t\right)
\leq
\max_{\substack{m<j\leq\lambda_m\\ n<k\leq\lambda_n}}\nu(x_{jk}-x_{mn},t)<\varepsilon.
\end{eqnarray*}
So conditions \eqref{tauber1} and \eqref{tauber2} are satisfied. Then, by Theorem \ref{theorem} we get $x_{mn}\to x$.
\end{proof}
\begin{definition}
A double sequence $(x_{mn})$ in $(V,\mu,\nu)$ is said to be slowly oscillating in the sense (1,0) if
\begin{equation*}
\sup_{\lambda>1}\liminf_{m,n\rightarrow\infty}\min_{m<j\leq\lambda_m}\mu(x_{jn}-x_{mn},t)=1
\end{equation*}
and
\begin{equation*}
\inf_{\lambda>1}\limsup_{m,n\rightarrow\infty}\max_{m<j\leq\lambda_m}\nu(x_{jn}-x_{mn},t)=0 ,
\end{equation*}
for all $t>0$.
\end{definition}
\begin{definition}
A double sequence $(x_{mn})$ in $(V,\mu,\nu)$ is said to be slowly oscillating in the sense (0,1) if
\begin{equation*}
\sup_{\lambda>1}\liminf_{m,n\rightarrow\infty}\min_{n<k\leq\lambda_n}\mu(x_{mk}-x_{mn},t)=1
\end{equation*}
and
\begin{equation*}
\inf_{\lambda>1}\limsup_{m,n\rightarrow\infty}\max_{n<k\leq\lambda_n}\nu(x_{mk}-x_{mn},t)=0 ,
\end{equation*}
for all $t>0$.
\end{definition}
\begin{theorem}\label{both}
If a double sequence $(x_{mn})$ in $(V,\mu,\nu)$ is slowly oscillating in the sense (1,0) and (0,1), then it is slowly oscillating in the sense (1,1).
\end{theorem}
\begin{proof}
By the facts that
\begin{eqnarray*}
\mu(x_{jk}-x_{mn},t)\geq\min\{\mu(x_{jk}-x_{mk},t/2),\mu(x_{mk}-x_{mn},t/2)\}
\\
\nu(x_{jk}-x_{mn},t)\leq\max\{\nu(x_{jk}-x_{mk},t/2),\nu(x_{mk}-x_{mn},t/2)\}
\end{eqnarray*}
and by slow oscillation of $(x_{mn})$ in the sense (1,0) and (0,1), we conclude that $(x_{mn})$ is slowly oscillating in the sense (1,1).
\end{proof}
By Theorem \ref{theorem}, Theorem \ref{slowtauber} and Theorem \ref{both}, we get following theorem.
\begin{theorem}\label{(1,0)(0,1)tauber}
Let $p,q\in SV\!A_+$ and $x\in V$. If double sequence $(x_{mn})$ in $(V,\mu,\nu)$ is slowly oscillating in the senses (1,0)\&(0,1) and $x_{mn}\to x\ (\bar{N},p,q;1,1)$, then $x_{mn}\to x$
\end{theorem}
In view of the thoerem above and Theorem \ref{boundedtoslowly}, we give the next theorem.
\begin{theorem}
Let $p,q\in SV\!A_+$ and $x\in V$. If $x_{mn}\to x\ (\bar{N},p,q;1,1)$ and sequences $\left\{m(x_{mn}-x_{m-1,n})\right\}$, $\left\{n(x_{mn}-x_{m,n-1})\right\}$ are q-bounded, then $x_{mn}\to x$.
\end{theorem}
We note that $(P_m)$ is regularly varying of index $\rho>0$ if\cite{chen}
\begin{eqnarray*}
\lim_{m\to\infty}\frac{P_{\lambda_m}}{P_m}=\lambda^{\rho}, \quad (\lambda>0).
\end{eqnarray*}
It is noted in \cite{chen}, $SV\!A_+$ contains all nonnegative sequences $p=(p_j)$ such that $(P_m)$ is regularly varying of positive index and
\begin{eqnarray}\label{varying}
\limsup_{m\to\infty}\frac{P_{\lambda_m}-P_m}{P_m}=\lambda^{\rho}-1.
\end{eqnarray}
\begin{theorem}\label{(1,0)}
Let $p=(p_j)$ be nonnegative sequence with $p_0>0$ and $(P_m)$ is regularly varying of positive indices. If $\left\{\frac{P_m}{p_m}(x_{mn}-x_{m-1,n})\right\}$ is q-bounded, then $(x_{mn})$ is slowly oscillating in the sense (1,0).
\end{theorem}
\begin{proof}
Let $\left\{\frac{P_m}{p_m}(x_{mn}-x_{m-1,n})\right\}$ be q-bounded.Then, for given $\varepsilon>0$ there exists $M_\varepsilon>0$ so that
\begin{equation*}
    t>M_\varepsilon \Rightarrow \inf_{m,n\in \mathbb{N}}\mu\left(\frac{P_m}{p_m}(x_{mn}-x_{m-1,n}),t\right)>1-\varepsilon \quad\textrm{and}\quad
    \sup_{m,n\in \mathbb{N}}\nu\left(\frac{P_m}{p_m}(x_{mn}-x_{m-1,n}),t\right)<\varepsilon.
\end{equation*}
For every $t>0$ choose $\lambda<\left(1+\frac{t}{M_\varepsilon}\right)^{1/\rho}$. Then for $n_0<m<j\leq\lambda_m$ we have
\begin{eqnarray*}
 \mu(x_{jn}-x_{mn},t)&=&\mu\left(\sum_{r=m+1}^{j}(x_{rn}-x_{r-1,n}),t\right)\\
 &\geq&\min_{m+1\leq r \leq j}\mu\left(x_{rn}-x_{r-1,n},\frac{p_r}{P_j-P_m}t\right)\\
  &=&\min_{m+1\leq r \leq j}\mu\left(\frac{P_r}{p_r}(x_{rn}-x_{r-1,n}),\frac{P_r}{P_j-P_m}t\right)\\
 &\geq&\min_{m+1\leq r \leq j}\mu\left(\frac{P_r}{p_r}(x_{rn}-x_{r-1,n}),\frac{P_m}{P_{\lambda_m}-P_m}t\right)\\
  &=&\min_{m+1\leq r \leq j}\mu\left(\frac{P_r}{p_r}(x_{rn}-x_{r-1,n}),\frac{t}{\frac{P_{\lambda_m}-P_m}{P_m}}\right)\\
   &\geq&\min_{m+1\leq r \leq j}\mu\left(\frac{P_r}{p_r}(x_{rn}-x_{r-1,n}),\frac{t}{\lambda^{\rho}-1}\right)\\
    &\geq&\inf_{m,n\in \mathbb{N}}\mu\left(\frac{P_m}{p_m}(x_{mn}-x_{m-1,n}),\frac{t}{\lambda^{\rho}-1}\right)\\
    &>&1-\varepsilon
\end{eqnarray*}
and
\begin{eqnarray*}
\nu(x_{jn}-x_{mn},t)<\sup_{m,n\in \mathbb{N}}\nu\left(\frac{P_m}{p_m}(x_{mn}-x_{m-1,n}),\frac{t}{\lambda^{\rho}-1}\right)<\varepsilon.
\end{eqnarray*}
in view of \eqref{varying}. Hence, $(x_{mn})$ is slowly oscillating in the sense (1,0).
\end{proof}
Similarly we can give next theorem for slow oscillation in the sense (0,1).
\begin{theorem}\label{(0,1)}
Let $q=(q_k)$ be nonnegative sequence with $q_0>0$ and $(Q_n)$ is regularly varying of positive indices. If $\left\{\frac{Q_n}{q_n}(x_{m,n}-x_{m,n-1})\right\}$ is q-bounded, then $(x_{mn})$ is slowly oscillating in the sense (0,1).
\end{theorem}
In view of Theorem \ref{(1,0)(0,1)tauber}, Theorem \ref{(1,0)} and Theorem \ref{(0,1)} we get following theorem.
\begin{theorem}
Let $p=(p_j)$ and $q=(q_k)$ be nonnegative sequences with $p_0>0, q_0>0$, and $(P_m)$ and $(Q_n)$ be regularly varying of positive indices. If $x_{mn}\to x\ (\bar{N},p,q;1,1)$ and sequences \linebreak$\left\{\frac{P_m}{p_m}(x_{mn}-x_{m-1,n})\right\}$,$\left\{\frac{Q_n}{q_n}(x_{m,n}-x_{m,n-1})\right\}$ are q-bounded, then $x_{mn}\to x$.
\end{theorem}
\section{Results for $(\bar{N},p,*;1,0)$ summability in $IFNS$}
Similar to the results of the previous section we can give results for $(\bar{N},p,*;1,0)$ summability of double sequences in $IFNS$. The proofs are also similar, hence they are omitted.

Similar to Theorem \ref{regular}, we give the the next theorem.
\begin{theorem}\label{regular1}
If double sequence $(x_{mn})$ in $(V,\mu,\nu)$ is q-bounded and convergent to $x\in V$, then $(x_{mn})$ is $(\bar{N},p,*;1,0)$ summable to $x$.
\end{theorem}
Similar to Theorem \ref{theorem}, we give the next theorem.
\begin{theorem}
Let $p\in SV\!A_+$ and double sequence $(x_{mn})$ be in $(V,\mu,\nu)$. Assume that $x_{mn}\to x\ (\bar{N},p,*;1,0)$. Then, $x_{mn}\to x$ if and only if for all $t>0$
\begin{eqnarray*}
\sup_{\lambda>1}\liminf_{m,n\rightarrow\infty}\mu\left(\frac{1}{P_{\lambda_m}-P_m}\sum_{j=m+1}^{\lambda_m}p_j(x_{jn}-x_{mn}),t\right)=1
\end{eqnarray*}
and
\begin{eqnarray*}
\inf_{\lambda>1}\limsup_{m,n\rightarrow\infty}\nu\left(\frac{1}{P_{\lambda_m}-P_m}\sum_{j=m+1}^{\lambda_m}p_j(x_{jn}-x_{mn}),t\right)=0.
\end{eqnarray*}
\end{theorem}
\begin{proof}
The proof is done in a similar way that in the proof of Theorem \ref{theorem} by using the equation(see \cite[Equation (3.13)]{chen})
\begin{eqnarray*}
\frac{1}{P_{\lambda_m}-P_m}\sum_{j=m+1}^{\lambda_m}p_j(x_{jn}-x_{mn})=t^{10}_{\lambda_m,n}-x_{mn}+\frac{1}{(P_{\lambda_m}/P_m)-1}(t^{10}_{\lambda_m,n}-t^{10}_{m,n})\qquad (\lambda>1; \lambda_m>m)
\end{eqnarray*}
instead of the equation \eqref{basiceq}.
\end{proof}
In view of the equation(see \cite[Equation (3.14)]{chen})
\begin{eqnarray*}
\frac{1}{P_{m}-P_{\lambda_m}}\sum_{j=\lambda_m+1}^{m}p_j(x_{jn}-x_{mn})=t^{10}_{m,n}-x_{mn}+\frac{1}{(P_{m}/P_{\lambda_m})-1}(t^{10}_{m,n}-t^{10}_{\lambda_m,n})\qquad (0<\lambda<1; \lambda_m<m)
\end{eqnarray*}
we can give the next theorem as analogue of Theorem \ref{theorem2}.
\begin{theorem}
Let $p\in SV\!A_+$ and double sequence $(x_{mn})$ be in $(V,\mu,\nu)$. Assume that $x_{mn}\to x\ (\bar{N},p,*;1,0)$. Then, $x_{mn}\to x$ if and only if for all $t>0$
\begin{eqnarray*}
\sup_{0<\lambda<1}\liminf_{m,n\rightarrow\infty}\mu\left(\frac{1}{P_m-P_{\lambda_m}}\sum_{j=\lambda_m+1}^{m}p_j(x_{mn}-x_{jn}),t\right)=1
\end{eqnarray*}
and
\begin{eqnarray*}
\inf_{0<\lambda<1}\limsup_{m,n\rightarrow\infty}\nu\left(\frac{1}{P_m-P_{\lambda_m}}\sum_{j=\lambda_m+1}^{m}p_j(x_{mn}-x_{jn}),t\right)=0.
\end{eqnarray*}
\end{theorem}
Similar to Theorem \ref{slowtauber}, we give the next theorem.
\begin{theorem}
Let $p\in SV\!A_+$ and $x\in V$. If double sequence $(x_{mn})$ in $(V,\mu,\nu)$ is slowly oscillating in the sense (1,0) and $x_{mn}\to x\ (\bar{N},p,*;1,0)$, then $x_{mn}\to x$.
\end{theorem}
In view of the thoerem above and Theorem \ref{boundedtoslowly}, Theorem \ref{(1,0)} we give the next theorems.
\begin{theorem}
Let $p\in SV\!A_+$ and $(x_{mn})$ be in $(V,\mu,\nu)$. If $x_{mn}\to x\ (\bar{N},p,*;1,0)$ and sequence $\left\{m(x_{mn}-x_{m-1,n})\right\}$ is q-bounded, then $x_{mn}\to x$.
\end{theorem}
\begin{theorem}
Let $p=(p_j)$ be a nonnegative sequences with $p_0>0$ and $(P_m)$ be regularly varying of positive indice. If $x_{mn}\to x\ (\bar{N},p,*;1,0)$ and sequence $\left\{\frac{P_m}{p_m}(x_{mn}-x_{m-1,n})\right\}$ is q-bounded, then $x_{mn}\to x$.
\end{theorem}

\end{document}